\documentclass[12pt]{article}
\usepackage{ct}
\usepackage{tikz}

\specs{n}{vol (issue)}{year}{}

\dateline{TBA}{TBA}{TBD}



\title{A proof of Frankl's conjecture on cross-union families}

\author[1]{
Stijn Cambie\thanks{Supported by IBS-R029-C4 and by the UK Research and Innovation Future Leaders Fellowship MR/S016325/1.}}
\author[2]{
Jaehoon Kim\thanks{Supported by the POSCO Science Fellowship of POSCO TJ Park Foundation and the National Research Foundation of Korea (NRF) grant funded by the Korea government(MSIT) No. RS-2023-00210430.}}

	\author[3]{Hong Liu\thanks{Supported by IBS-R029-C4 and by the UK Research and Innovation Future Leaders Fellowship MR/S016325/1.}}
	\author[4]{Tuan Tran\thanks{
		Supported by the Institute for Basic Science (IBS-R029-Y1), and the Outstanding Young Talents Program (Overseas) of the National Natural Science Foundation of China.}}

\affil[2]{%
Department of Mathematical Sciences, KAIST, South Korea.

\email{jaehoon.kim@kaist.ac.kr}%
}

\affil[4]{School of Mathematical Sciences, University of Science and Technology of China, China.
		\email{trantuan@ustc.edu.cn}}

\affil[1,3]{%
Extremal Combinatorics and Probability Group (ECOPRO), Institute for Basic Science (IBS), Daejeon, South Korea
\email{stijn.cambie@hotmail.com}, \email{hongliu@ibs.re.kr}%
}

\def\G{\mathcal{G}}
\def\mH{\mathcal{H}}
\def\F{\mathcal{F}}

\newenvironment{poc}{\begin{proof}[Proof of claim]}{\end{proof}}

\begin{document}

\maketitle


\begin{abstract}
The families $\F_0,\ldots,\F_s$ of $k$-element subsets of $[n]:=\{1,2,\ldots,n\}$ are called {\em cross-union} if there is no choice of $F_0\in \F_0, \ldots, F_s\in \F_s$ such that $F_0\cup\ldots\cup F_s=[n]$. A natural generalization of the celebrated Erd\H{o}s--Ko--Rado theorem, due to Frankl and Tokushige, states that for $n\le (s+1)k$ the geometric mean of $\lvert\F_i\rvert$ is at most $\binom{n-1}{k}$. Frankl conjectured that the same should hold for the arithmetic mean under some mild conditions. We prove Frankl's conjecture in a strong form by showing that the unique (up to isomorphism) maximizer for the arithmetic mean of cross-union families is the natural one $\F_0=\ldots=\F_s={[n-1]\choose k}$.
\end{abstract}



\section{Introduction}

	The most natural operations on sets are intersections and unions. These two seemingly simple operations surprisingly give rise to exciting theories on collections of sets. The most famous such instances in extremal set theory are the theory on intersecting families and the theory of hypergraph matchings. The Erd\H{o}s-Ko-Rado theorem~\cite{EKR61} is arguably the most foundational result in the former regime, while the Erd\H{o}s matching conjecture~\cite{Erdos65} is the most central theme in the latter. In this paper, we consider a problem of Frankl which has deep connections to both of these intricate theories.
	
     Let us start by recalling the cornerstone Erd\H{o}s-Ko-Rado theorem. We say a family $\F$ of sets is \emph{intersecting} if $A\cap B\neq \varnothing$ for any $A,B\in \F$. 
	\begin{theorem}[\cite{EKR61}]\label{thr:EKR61}
		Let $n$ and $k$ be two positive integers with $n \ge 2k$. If $\F \subset \binom{[n]}{k}$ is an intersecting family, then $\lvert \F \rvert \le \binom{n-1}{k-1}.$ 
		This bound is sharp as equality holds if $\mathcal{F} = \{A \in {[n] \choose k} \colon 1 \in A\}$.
	\end{theorem}	
This concept of intersecting family further generalizes to \emph{$(s+1)$-cross-intersecting} families, which is a collection $\F_0,\dots, \F_s$ of families of sets where $\bigcap_{0\leq i\leq s} A_i\neq \emptyset$ for any choice of $A_i\in \F_i$ for all $0\leq i\leq s$.
    There have been numerous interesting generalizations of the Erd\H{o}s-Ko-Rado theorem to cross-intersecting families.
    The maximum value of $\prod_{0\leq i\leq s} |\F_i|$ was considered in e.g.~\cite{Bey05, Pyber86, MT89a, MT89b, FT11, FLST14, Borg15, Borg16, Borg17} and the maximum value of $\sum_{0\leq i\leq s}|\F_i|$ was considered in \cite{HiltonMilner67, Hilton77, Borg14, BF20, WZ13, WZ11}.
	
	In the case of $(s+1)$-cross intersecting families, if $n< (s+1)k/s$, then trivially $\F_0=\dots =\F_s = \binom{[n]}{k}$ provides the maximum possible collection.	
	For $n\ge (s+1)k/s$, the most natural example for the maximum product is the collection with $s+1$ copies of $ \{ A\in \binom{[n]}{k}: 1\in A  \}$, and this indeed is extremal as shown in \cite{FT11}.
	The sum version is more delicate, as certain relations of $s,k$ and $n$ might yield a different maximum as in the case of \cite{HiltonMilner67, BF20}. For a simple example, when $s=1$ and $n> 2k\geq 4$, a very asymmetric collection $\F_0=\{[k]\}$ and $\F_1=\{A \in \binom{[n]}{k}: A\cap [k]\neq \varnothing\}$ provides a maximum sum when the families are required to be non-empty.
	To better illustrate the relations among $s,k$ and $n$, it is much more convenient to consider the complements of the sets rather than the sets itself.
	
	By considering complements of the sets in an $(s+1)$-cross-intersecting family, we obtain the following notion. 
	\begin{definition}
A collection $\F_0,\dots, \F_s$ of families of nonempty sets in $\binom{[n]}{k}$ is \emph{$(s+1)$-cross-union} (or simply \emph{cross-union}) if $\bigcup_{0\leq i\leq s} A_i\neq [n]$ for any choice of $A_i\in \F_i$ for all $0\le i \le s.$
	\end{definition}	

Here, we only consider the case where the families are nonempty.
With this definition, we are interested in values of $(n,k,s)$ which ensure that $\F_0=\dots =\F_s = \binom{[n-1]}{k}$ is a 
cross-union collection maximizing the sum $\sum_{0\leq i\leq s} |\F_i|$. 
Indeed, Frankl proposed the following conjecture in \cite{Frankl21}.

\begin{conjecture}[Frankl, \cite{Frankl21}]\label{conj:Frankl_crossunion}
	Let $k \ge 2$ and $1 \le \ell \le k.$
	There exists $s_0=s_0(\ell)\ge 2$ such that for each $s \geq s_0$, if $n = sk + \ell$ and $\mathcal{F}_0, \mathcal{F}_1, \dots, \mathcal{F}_s$ are non-empty cross-union subfamilies of ${[n] \choose k}$, then
	$$
	\frac{\lvert \mathcal F_0 \rvert+ \lvert \mathcal F_1 \rvert + \ldots +  \lvert \mathcal F_s \rvert}{s+1} 
	\le  \binom{n-1}{k}.$$
\end{conjecture}
Here, the assumption that $n\leq (s+1)k$ is necessary as otherwise the union of $(s+1)$ sets of size $k$ is never equal to $[n]$.
On the other hand, the assumption $n>sk$ is also very natural. Indeed, note that 
$\F_0,\dots, \F_{s+1}$ being cross-union implies that $\F_0,\dots, \F_{s}$ is also cross-union.
Hence, assuming that $\F_{s+1}$ is the smallest among the families $\F_0,\dots, \F_{s+1}$, we obtain 
\[ \frac{\lvert \mathcal F_0 \rvert+ \lvert \mathcal F_1 \rvert + \ldots +  \lvert \mathcal F_{s+1} \rvert}{s+2} \leq 
	\frac{\lvert \mathcal F_0 \rvert+ \lvert \mathcal F_2 \rvert + \ldots +  \lvert \mathcal F_{s} \rvert}{s+1}.\]
Therefore, proving the above conjecture for $s=\lfloor \frac{n}{k} \rfloor$ yields the results for all larger values of $s$ and hence the condition $sk< n \leq (s+1)k$ in Conjecture~\ref{conj:Frankl_crossunion} is sensible.

We further remark that the condition $s\ge s_0(\ell)$ in Conjecture~\ref{conj:Frankl_crossunion} is also necessary.
Indeed, for small values of $s$, the conclusion of Conjecture~\ref{conj:Frankl_crossunion} does not always hold. 	For example, Hilton and Milner~\cite{HiltonMilner67} proved that for $s=1$, the maximum of $\frac{1}{s+1}\sum_{0\leq i\leq s} |\F_i|$ is not  $\binom{n-1}{k}$. Moreover, the following example shows that the value $s_0$ must depend on $\ell$. 
\begin{example}\label{examp:example}
For $s\geq 2, \ell\geq 1,c\ge 1,k=\ell+c$ and $n=sk+\ell$,
the families $\F_0= \{[k]\}$, $\F_1 =\{A\in \binom{[n]}{k}: |A\cap [k]|\geq c+1\}$ and $\F_2=\dots= \F_s= \binom{[n]}{k}$ are cross-union.
\end{example}\noindent
In fact, this example shows that for fixed $c$ the condition $s_0 = \Omega\left(\frac{\ell}{\ln \ell}\right)$ is necessary. 
We know that $\binom{k}{\leq c} \leq (c+1) k^c$ and $\binom{n-1}{k}=\frac{n-k}{n}\binom{n}{k}$.
If $k\ge 3$ and $s< \frac{k}{ (c+2) \ln{k}}-1$,
		then $\frac{ \binom{n-k}{k} }{ \binom{n}{k} } \le \left( \frac{n-k}{n} \right)^k \le \left( 1-\frac{1}{s+1} \right)^k = \left( \frac{s}{s+1} \right)^k< e^{- k/(s+1)} \leq \frac{1}{k^{c+2}} \leq \frac{1}{(c+2)n k^c}.$ 
Hence, in this case, \cref{examp:example} satisfies
		\begin{align*}
		\sum_{0\leq i\leq s} |\F_i|   &\geq  1+s\binom{n}{k}- \binom{k}{\leq c} \binom{n-k}{k}
		\ge s\binom{n}{k} - (c+1) k^c \left( \frac{s}{s+1} \right)^k \binom{n}{k}\\
		&> \left(s-\frac{c}{n}\right) \binom{n}{k}=(s+1)\frac{n-k}{n}\binom{n}{k}
		=(s+1)\binom{n-1}{k}.
		\end{align*}		
Towards \cref{conj:Frankl_crossunion}, Frankl \cite{F21} proved sporadic cases.
The main 
result in this paper is the following theorem, verifying a strong form of Conjecture~\ref{conj:Frankl_crossunion} and yielding the uniqueness of the extremal families. 

\begin{theorem}\label{thr:main}
	Let $n=sk+\ell$ with $1 \le \ell \le k$ and $s \ge 4\ell$.
	Suppose that $\mathcal F_0, \mathcal F_1, \dots, \mathcal F_s \subset \binom{[n]}{k}$ are non-empty and cross-union. Then 
	$$\frac{\lvert \mathcal F_0 \rvert+ \lvert \mathcal F_1 \rvert + \ldots +  \lvert \mathcal F_s \rvert}{s+1} 
	\le  \binom{n-1}{k}.
	$$
	Furthermore, equality is attained only if $\mathcal F_0=\ldots=\mathcal F_s=\binom{[n]\setminus\{i\}}{k}$ for some $i\in [n]$.
\end{theorem}

In view of \cref{examp:example}, the linear bound $s\ge 4\ell$ above is best possible up to a logarithmic factor.

We remark that Conjecture~\ref{conj:Frankl_crossunion} has a clear connection with the Erd\H{o}s matching conjecture. A collection of $s$ sets in $[n]$ is a \emph{matching of size $s$} if they are pairwise disjoint.

\begin{conjecture}[The Erd\H{o}s matching conjecture~\cite{Erdos65}]
If $n \ge k(s+1)$ and $\F \subset \binom{[n]}{k}$ has no matching of size $s+1$, then $$
	\lvert \F \rvert \le \max \left \{ \binom{n}{k}-\binom{n-s}{k}, \binom{k(s+1)-1}{k} \right \}.$$
\end{conjecture}

The Erd\H{o}s matching conjecture has been known to be true for $n$ sufficiently large in terms of $s$ and $k$ since the publication of Erd\H{o}s's paper~\cite{Erdos65}. Frankl~\cite{Frankl17} showed that the conjecture is also true if $n=k(s+1)+\ell$ for the range $0 \le \ell \le \varepsilon(k)(s+1)$, where $\varepsilon(k)>0$ is a constant only depending on $k$. There have been many interesting works \cite{BDE76,HLS12,F13, FK18b} that improved the range of $n$ for which the conjecture is known to hold.

In the case where $n=k(s+1)$ and $\F_0=\dots=\F_s=\F$, the collection $\{\F_0,\dots,\F_s\}$ is cross-union if and only if $\F$ has no matching of size $s+1$. From this, one can naturally consider several `cross' versions of the Erd\H{o}s matching conjecture.
We will discuss some variants of the `cross' version of the Erd\H{o}s matching conjecture in Section~\ref{sec:conc}.


\section{Preliminaries}

For a set family $\F$, the shadow of $\F$ at level $s$ is defined by
	$$
	\sigma_{s}(\F)=\{G\colon |G|=s, \exists F\in \F \enskip \text{with} \enskip G\subset F\}.
	$$
	The following theorem by Frankl~\cite[Theorem~11.1]{F87} will be useful. 
	A family $\F$ is \emph{$r$-wise union} if $\bigcup_{1\leq i\leq r} A_i\neq [n]$ for every choice of sets $A_1, \ldots, A_r \in \F.$
\begin{theorem}[\cite{Fr76,F87}]\label{thr:r-wise}
Let $n, k$ and $r$ be positive integers with $r \ge 2$ and $n \le rk$. If $\F \subset \binom{[n]}{k}$ is an $r$-wise union family, then $\lvert \F \rvert \le \binom{n-1}{k}.$
Moreover, except for $r=2$ and $n=2k$, equality is attained only if $\F=\binom{[n] \backslash \{i\}}{k}$ for some $i \in [n].$
\end{theorem}

\subsection{Combinatorial lemmas}
In this section, we collect several combinatorial results that are needed for the proof of \cref{thr:main}.
A basic result of Frankl \cite{F87} (see \cref{lem:shifting} below) allows us to restrict ourself to {\em shifted} families. We say that a family $\F \subset \binom{[n]}{k}$ is shifted if for any $F=\{x_1,\ldots,x_k\} \in \F$ and any $G=\{y_1,\ldots,y_k\} \subset [n]$ such that $y_i \le x_i$ for every $1\le i \le k$, we have $G\in \F$. It is easy to see that if $\F \subset \binom{[n]}{k}$ is non-empty and shifted, then $[k]\in \F$.

\begin{lemma}[\cite{F87}]\label{lem:shifting}
Suppose that the families $\F_0,\ldots,\F_s \subset \binom{[n]}{k}$ are cross-union. Then there exist shifted and cross-union families $\F'_0,\ldots,\F'_s \subset \binom{[n]}{k}$ such that $\lvert \F_i\rvert =\lvert \F'_i\rvert$ for $0\le i \le s$.
\end{lemma}


The second lemma is a probabilistic version of Katona's circle method.
 
\begin{lemma}\label{lem:Katonascircle_prob}
		Let $k_0, k_1, \ldots, k_s, n$ be positive integers with $k_0+k_1+\ldots+k_s\ge n.$
		Suppose that $\G_0 \subset \binom{[n]}{k_0},\G_1 \subset \binom{[n]}{k_1}, \ldots, \G_s \subset \binom{[n]}{k_s}$ are cross-union.
		Then $$\sum_{i=0}^s \frac{ \lvert \G_i \rvert }{\binom{n}{k_i}} \le s.$$
If $s\ge 2,$ $k_0=\ldots=k_s=k$, $n=(s+1)k$, and $\emptyset \ne \G_0\subset \G_1\subset \ldots \subset \G_s$, then the equality holds only if $\G_0=\G_1=\ldots=\G_s$. 
\end{lemma}

\noindent {\bf Remark.} The first part of \cref{lem:Katonascircle_prob} is a result of Frankl \cite[Lemma 2.4]{Frankl21}.

\begin{proof}[Proof of \cref{lem:Katonascircle_prob}]
Fix $s+1$ sets $A_0, \ldots, A_s$ satisfying $|A_0|=k_0,\ldots,|A_s|=k_s$ and $A_0\cup \ldots \cup A_s=[n]$. 
Let $(\Omega, \mathbb{P})$ be the probability space where $\Omega$ is the set of permutations of $[n]$ and $\mathbb{P}$ is the uniform measure on $\Omega$.
Let 
$X_i \colon \Omega \rightarrow \{0,1\}$ be the random variable 
defined by letting $X_i(\alpha)=1$ if $\alpha(A_i) \in \G_i$ and $X_i(\alpha)=0$ otherwise.
Let $X= \sum_{i=0}^s X_i.$
Choose a permutation $\alpha \in \Omega$ of $[n]$ uniformly at random. Since $\mathbb{P}(\alpha(A_i) \in \G_i)=\lvert \G_i \rvert/\binom{n}{k_i}$, we have $\mathbb{E}[X]=\sum_{i=0}^s \mathbb{E}[X_i] =\sum_{i=0}^s \lvert \G_i \rvert /\binom{n}{k_i}$, by linearity of expectation. On the other hand, the cross-union property implies $X \le s$, resulting in $\mathbb{E}[X] \le s$. Therefore, $\sum_{i=0}^s \lvert \G_i \rvert /\binom{n}{k_i} \le s$.

Now we deal with the equality part of the theorem. To ease the notation, let $\F_i= \binom{[n]}{k} \backslash \G_i$ for $0\le i\le s$. Since $\G_0\subset \G_1\subset \ldots\subset \G_s$, we have $\F_s \subset \F_{s-1} \subset \ldots \subset \F_0$. 

\begin{claim}\label{claim:partition1}
Let $[n]=B_0\cup \ldots \cup B_s$ be a partition of $[n]$ into $s+1$ sets of size $k$ each. Then there is exactly one $i \in \{0,1,\ldots,s\}$ for which $B_i \in \F_i$.
\end{claim}

\begin{poc}
Since $k_0=k_1=\ldots=k_s=k$, in the first part of the proof of \cref{lem:Katonascircle_prob} we have $[n]=A_0\cup \ldots \cup A_s$, $n=(s+1)k$ and $|A_0|=\ldots=|A_s|=k$, hence $[n]=A_0\cup \ldots \cup A_s$ is a partition of $[n]$ into $s+1$ sets of size $k$. Let $\alpha$ be a permutation of $[n]$ with $\alpha(A_i)=B_i$ for every $0\le i \le s$. We can infer from the first part of the proof of \cref{lem:Katonascircle_prob} that there is exactly one $i \in \{0,1,\ldots,s\}$ for which $\alpha(A_i) \in \F_i$. As $\alpha(A_i)=B_i$, this completes our proof.
\end{poc}

\begin{claim}\label{claim:F1-s}
$\F_1=\ldots=\F_s$.
\end{claim}

\begin{poc}
Since $\F_s\subset \F_{s-1}\subset \ldots \subset \F_0$, it suffices to show that $B_1\in \F_s$ whenever $B_1 \in \F_1$. Fix $B_1\in \F_1$, and consider a partition $[n]=B_0\cup B_1\cup \ldots\cup B_s$ of $[n]$ into size-$k$ sets. 

Given $j\in \{0,2,3,\ldots,s\}$, let $\pi$ be a permutation of $\{0,1,\ldots,s\}$ with $\pi(0)=j$ and $\pi(1)=1$. Applying \cref{claim:partition1} to the partition $[n]=B_{\pi(0)}\cup B_{\pi(1)}\cup \ldots\cup B_{\pi(s)}$ and noting that $B_{\pi(1)}=B_1\in \F_1$, we find $B_j=B_{\pi(0)} \notin \F_0$. Hence $B_0,B_2,\ldots,B_s$ do not belong to $\F_0=\F_0\cup\F_1\cup\ldots\cup \F_s$.

Consider a permutation $\tau$ of $\{0,1,\ldots,s\}$ with $\tau(s)=1$. By \cref{claim:partition1}, there exists $i\in \{0,1,\ldots,s\}$ such that $B_{\tau(i)}\in \F_i$. Since $B_0,B_2,\ldots,B_s \notin \F_0\cup\F_1\cup\ldots\cup \F_s$, we must have $\tau(i)=1$. It follows that $i=s$, and so $B_1 \in \F_s$, as required.
\end{poc}

\begin{claim}\label{claim:partition2}
Let $[n]=B_0\cup \ldots \cup B_s$ be a partition of $[n]$ into $s+1$ sets of size $k$ each. If $B_j\in \F_0\backslash\F_1$ for some $j$, then $B_0,\ldots,B_s\in \F_0\backslash\F_1$.
\end{claim}

\begin{poc}
Without loss of generality we can assume $B_0\in \F_0\backslash\F_1$. To prove the claim it suffices to show $B_1\in \F_0\backslash\F_1$. Applying \cref{claim:partition1} to the partition $[n]=B_0\cup B_1\cup\ldots\cup B_s$ and noting that $B_0\in \F_0$, we get $B_i\notin \F_i=\F_1$ for every $i\ge 1$. Again, we apply \cref{claim:partition1} to the partition $[n]=B_1\cup B_0\cup B_2\cup\ldots\cup B_s$ and find $B_1\in\F_0$. Therefore, we obtain $B_1\in \F_1\backslash\F_0$, as desired.
\end{poc}

\begin{claim}\label{claim:F0-1}
$\F_0=\F_1$.    
\end{claim}

\begin{poc}
Suppose to the contrary that $\F_0 \backslash \F_1 \ne \emptyset$.
Note that $\F_1 \ne \emptyset$, for otherwise we would have $\G_1=\ldots=\G_s=\binom{[n]}{k}$ and $\G_0=\emptyset$. Let $B_0\in \F_0\backslash\F_1$ and $C\in \F_1$. 

Consider a partition $[n]=B_0\cup B_1\cup\ldots\cup B_s$ of $[n]$ into size-$k$ sets. Because $B_0\in \F_0\backslash\F_1$, we can deduce from \cref{claim:partition2} that $B_0,B_1,\ldots,B_s\in \F_0\backslash\F_1$.
Since $|(B_0 \cup B_1) \backslash C| \ge |B_0\cup B_1|-|C|=k$, one can find a size-$k$ subset $B'_0 \subset (B_0 \cup B_1) \backslash C$. Let $B'_1=(B_0 \cup B_1) \backslash B'_0.$
Notice that $|B'_0|=|B'_1|=k$ and $B'_0\cup B'_1=B_0\cup B_1$. Thus, $[n]=B'_0\cup B'_1 \cup B_2 \cup \ldots \cup B_s$ is a partition of $[n]$ into sets of size $k$. As $B_s\in \F_0\backslash\F_1$, an application of \cref{claim:partition2} gives $B'_0,B'_1,B_2,\ldots,B_s\in \F_0 \backslash \F_1$. But now $B'_0 \in \F_0$ and $C \in \F_1$ are disjoint and one can extend to a partition $[n]=C \cup B'_0 \cup B'_2 \cup \ldots B'_s$, which contradicts \cref{claim:partition1}. This is sketched in Figure~\ref{fig:difpartitions} for $s=2.$
\end{poc}

The equality part follows from \cref{claim:F1-s} and \cref{claim:F0-1}. 
\end{proof}


\begin{figure}[h]
	\begin{center}
	\begin{tikzpicture}[scale=0.82]

\draw[fill=black!10!white] 	(0,0) rectangle (3,1);
\draw[fill=black!20!white] 	(6,0) rectangle (3,1);
\draw[fill=black!30!white] 	(6,0) rectangle (9,1);

\foreach \x in {0,3,6}
{
\draw[fill=black!50!white] 	(\x,1.5) rectangle (\x+1,2.5);
}

\draw[fill=black!20!white] 	(0,-1.5) rectangle (1,-0.5);
\draw[fill=black!20!white] 	(3,-1.5) rectangle (5,-0.5);

\draw[fill=black!10!white] 	(1,-1.5) rectangle (3,-0.5);
\draw[fill=black!10!white] 	(6,-1.5) rectangle (5,-0.5);

\draw[fill=black!30!white] 	(6,-1.5) rectangle (9,-0.5);
\draw[fill=black!10!white] 	(1,-2) rectangle (3,-3);
\draw[fill=black!10!white] 	(6,-2) rectangle (5,-3);
\draw[fill=black!30!white] 	(7,-3) rectangle (9,-2);
\draw[fill=black!30!white] 	(5,-3) rectangle (4,-2);

\foreach \x in {0,3,6}
{
\draw[fill=black!50!white] 	(\x,-2) rectangle (\x+1,-3);
}
\coordinate [label=center:$C$] (A) at (-1.5,2); 

\coordinate [label=center:$B_0 \cup B_1 \cup B_2$] (A) at (-1.5,0.5);
\coordinate [label=center:$ B_2$] (A) at (7.5,0.5);
\coordinate [label=center:$B'_0 \cup B'_1 \cup A_2$] (A) at (-1.5,-1);
\coordinate [label=center:$B'_0$] (A) at (2,-1);
\coordinate [label=center:$B'_0$] (A) at (5.5,-1);
\coordinate [label=center:$C \cup B'_0 \cup B'_2$] (A) at (-1.5,-2.5);

	\end{tikzpicture}
\end{center}
	\caption{Different partitions of $[n]$ for $s=2$ and the set $C$}
	\label{fig:difpartitions}
\end{figure}
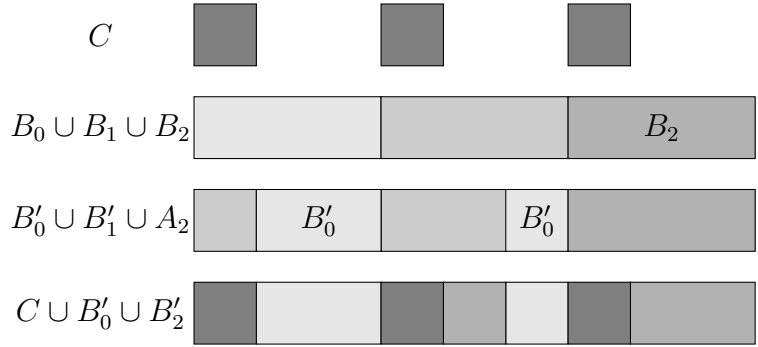


We will require the following slightly weaker version of the Kruskal–Katona theorem, due to Lov\'asz~\cite{Lovasz76}.
Here $\binom{x}{k}=\frac{x \cdot(x-1)\cdot \ldots \cdot (x-k+1)}{k!}$ is defined for every real number $x$ and integer $k$.
\begin{theorem}\label{thr:Lovasz}
Let $k\ge \ell>0$ be two integers and $x \ge k$ a real number. If $\F \subset \binom{[n]}{k}$ and $|\F|=\binom{x}{k}$, then $|\sigma_{\ell}(\F)| \ge \binom{x}{\ell}$.
\end{theorem}
\subsection{Technical lemmas}\label{subsec:comp}
The following two technical lemmas will be used.
\begin{lemma}\label{lem:computation}
Let $n=ks+\ell$ with $1 \le \ell\le k$ and $s\ge 4\ell$. The following holds.
\begin{itemize}
    \item[\rm (i)] If $k\ge 2\ell$, then $(s+1)\binom{n-1}{k}-s\binom{n}{k}+\binom{ks}{k} \ge \frac{\ell}{k}\binom{n}{k}$.
    \item[\rm (ii)] If $k<2\ell$, then $(s+1)\binom{n-1}{k}-s\binom{n}{k}+\binom{ks}{k} \ge \binom{(1-1/k)n+1}{k}$.
\end{itemize}
\end{lemma}
\begin{proof}
(i) For $x\ge ks$, we have
		$$\binom{x-1}{k}=\frac{x-k}{x}\binom{x}{k} \ge \left( 1- \frac 1s \right) \binom{x}{k}.$$
		By iterating this and noting that $n-ks=\ell$, we obtain $\binom {ks}{k} \ge  \left( 1- \frac 1s \right)^\ell \binom{n}{k} \ge \left( 1- \frac {\ell}s \right) \binom{n}{k}$, where the second inequality is true by Bernoulli's inequality. Since $n\ge ks$, we obtain $\binom{n-1}{k} \ge \left( 1- \frac {1}s \right) \binom{n}{k}$.
		Therefore,
		$$(s+1)\binom{n-1}{k}-s\binom{n}{k}+\binom{ks}{k} \ge
		\left(1- \frac {\ell+1}{s}\right) \binom{n}{k} \ge \frac{\ell}{k}\binom{n}{k}$$
		assuming $k\ge 2\ell$ and $s\ge 4\ell$.
		
(ii) Since $2\ell\ge k+1$ and $s\ge 4\ell$, we have $n\ge ks \ge 2k^2+2k$. Thus $$\frac{\binom{n-1}{k-1}}{\binom{n-2k}{k-1}} \le \left( \frac{n-k+1}{n-3k+2}\right)^{k-1}\le \left( 1+\frac{1}{k} \right)^{k-1}\le k.$$		
It follows that 
\begin{align*}
(s+1)\binom{n-1}{k}-s\binom{n}{k}+\binom{ks}{k}&\ge\binom{ks}{k} -\binom{n-1}{k-1}\\
& \ge \binom{n-k}{k}-k\binom{n-2k}{k-1}\\
&\ge \binom{n-2k}{k}\\
&\ge \binom{(1-1/k)n+1}{k},
\end{align*}
where in the first line we used $(s+1)\binom{n-1}{k}=\left(s-\frac{k-\ell}{n}\right)\binom{n}{k}=s\binom{n}{k}-\binom{n-1}{k-1}+\frac{\ell}{n}\binom{n}{k}$,
the third inequality holds since $\binom{n-k}{k}-\binom{n-2k}{k}=\sum_{m=n-2k}^{n-k-1}\binom{m}{k-1} \ge k\binom{n-2k}{k-1}$, and in the last inequality we used $n\ge 2k^2+2k$.
\end{proof}

\begin{lemma}\label{lem:different-slices}
 Let $k, \ell$ and $ n$ be integers with $1 \le \ell\le k<n$.
		Let $x_0 \in [k,n-1]$ be a real number for which 
		\begin{equation}\label{eq:assumption}
		\frac{\binom{x_0}{\ell}}{\binom{n}{\ell}}  \le \frac{k}{\ell} \frac{\binom{x_0}{k}}{\binom{n}{k}}.
		\end{equation}
		Then
		$$\frac{ \binom{x_0}{\ell}}{\binom{n}{\ell}}\ge \frac{ \binom{x_0}{k}}{\binom{n}{k}} + \frac{k-\ell}{n}.$$   
    Furthermore, the equality occurs if and only if either $\ell=k$, or $\ell<k$ and $x_0=n-1.$
\end{lemma}
\begin{proof}
We write $A(x)=\frac{\binom{x}{k}}{\binom{n}{k}}$ and $B(x)=\frac{\binom{x}{\ell}}{\binom{n}{\ell}}$. Consider the function $f(x)=B(x)-A(x)$, where $x_0 \le x \le n-1$. We wish to show $f(x_0) \ge f(n-1)=\frac{k-\ell}{n}$. 	

Notice first that 
\begin{equation}\label{eq:derivative}
f'(x)=B(x)\left( \frac{1}{x}+\frac{1}{x-1} + \ldots + \frac{1}{x-\ell+1} \right)  -  A(x) \left( \frac{1}{x}+\frac{1}{x-1} + \ldots + \frac{1}{x-k+1} \right).    
\end{equation}

By~\eqref{eq:assumption}, we have $\frac{A(x_0)}{B(x_0)}\ge \frac{\ell}{k}$. Hence	
	\begin{align}\label{eq:comp_ratiobinomials}
	\frac{A(x)}{B(x)}=
	\prod_{i=\ell}^{k-1} \frac{x-i}{n-i}&\ge \prod_{i=\ell}^{k-1} \frac{x_0-i}{n-i}
	=\frac{A(x_0)}{B(x_0)}\ge \frac{\ell}{k}.
	\end{align}
	
As $\frac{1}{x} \le \frac{1}{x-1} \le   \ldots \le \frac{1}{x-\ell+1} \le \ldots \le \frac{1}{x-k+1}$, we see that
	\begin{equation}\label{eq:comp_sumreciprosals}
	\frac{1}{x}+\frac{1}{x-1} + \ldots + \frac{1}{x-k+1}\ge \frac{k}{\ell}\left(\frac{1}{x}+\frac{1}{x-1}+ \ldots + \frac{1}{x-\ell+1}\right).
	\end{equation}

	From \eqref{eq:derivative}, \eqref{eq:comp_ratiobinomials} and \eqref{eq:comp_sumreciprosals}, we conclude
	$f'(x)\le 0$ for every $x\in [x_0,n-1]$. Thus $f(x_0)\ge f(n-1)=\frac{k-\ell}{n}$, as desired.

Now assume $\ell<k$ and $x_0<n-1.$ Since the central inequality in \eqref{eq:comp_ratiobinomials} is strict for  $\ell<k$ and $x_0<x<n-1,$ we have $f'(x)<0$ and thus $f(x_0)> f(n-1)=\frac{k-\ell}{n}$, i.e., the inequality is strict.
\end{proof}


\section{Proof of Frankl's conjecture}

We are now ready to prove  \cref{thr:main}. 
\begin{proof}[Proof of \cref{thr:main}]
Let $\F_0,\F_1,\ldots,\F_s\subset{[n]\choose k}$ be non-empty cross-union families maximizing the sum $\sum_{i=0}^s|\F_i|$. Considering $\F_0=\F_1=\ldots=\F_s={[n]\setminus\{1\}\choose k}$, we may assume that $\sum_{i=0}^s|\F_i|\ge (s+1){n-1\choose k}$.

We first show that we must have equality $\sum_{i=0}^s|\F_i|= (s+1){n-1\choose k}$. To this end, it suffices to consider families that are nested via the following claim.

\begin{claim}[\cite{Frankl21}]\label{claim:shifted-nested}
There exist nested families $\emptyset\ne \G_0 \subset \G_1\subset\ldots \subset \G_s \subset \binom{[n]}{k}$ such that the collection $\{\G_0,\ldots,\G_s\}$ is cross-union and $\sum_{i=0}^s|\G_i|=\sum_{i=0}^s|\F_i|$. Furthermore, if $|\F_0|,\ldots,|\F_s|$ are not all equal, then $\G_0,\ldots,\G_s$ are not all the same. 
\end{claim}
\begin{poc}
From \cref{lem:shifting}, we can assume further that $\F_0, \F_1, \ldots, \F_s$ are non-empty and shifted. In particular, $[k] \in \F_i$ for every $0\le i \le s$. For a fixed pair $0\le u <v \le s$, replacing $\F_u$ and $\F_v$ by $\F_u\cap \F_v$ and $\F_u\cup \F_v$ will preserve the nonemptiness, the cross-union property, and the sum $\sum_{i=0}^s\lvert \F_i\rvert$. Iterating this operation for all pairs $0\le u<v\le s$ (in lexicographical order) will generate $s+1$ nested families with the desired properties. The `furthermore' part follows from the fact that if $\F_u\ne \F_v$, then $|\F_u\cap \F_v|< |\F_u\cup \F_v|$.      
\end{poc}

Let $\G_0,\G_1\ldots,\G_s$ be the nested families given by \cref{claim:shifted-nested}. Then, $\sum_{i=0}^s\lvert\G_i\rvert =\sum_{i=0}^s|\F_i|\ge (s+1)\binom{n-1}{k}$, or equivalently,
\begin{equation}\label{eq:contradiction}
\sum_{i=0}^s\frac{\lvert\G_i\rvert}{\binom{n}{k}} \ge \frac{(s+1)\binom{n-1}{k}}{\binom{n}{k}}=s-\frac{k-\ell}{n}.  
\end{equation}


Since $\G_0\subset \G_i$ for $1\le i \le s$, $\G_0$ is $(s+1)$-wise union. By \cref{thr:r-wise},
$\lvert \G_0\rvert \le \binom{n-1}{k}$. So we can write $\lvert\G_0\rvert=\binom{x_0}{k}$ for some $x_0\in [k, n-1]$.     

Since the families $\G_0,\G_1,\ldots,\G_s$ are non-empty and cross-union, so are the families $\sigma_{\ell}(\G_0) ,\G_1,\ldots,\G_s$.
Thus Lemma~\ref{lem:Katonascircle_prob} applies. We conclude 
\begin{equation}\label{eq:Katonacircle-bound}
\frac{ \lvert \sigma_{\ell}(\G_0) \rvert}{\binom{n}{\ell}} + \sum_{i=1}^{s}  \frac{\lvert \G_i\rvert}{\binom{n}{k}} \le s.     
\end{equation}
Furthermore, as $\lvert\G_0 \rvert=\binom{x_0}{k}$ with $x_0 \ge k$, \cref{thr:Lovasz} implies 
\begin{equation}\label{eq:Lovasz}
\lvert \sigma_{\ell}(\G_0) \rvert \ge \binom{x_0}{\ell}.    
\end{equation}

\noindent We claim that
\begin{equation}\label{eq:different-slices}
\frac{\binom{x_0}{\ell}}{\binom{n}{\ell}} \ge \frac{\binom{x_0}{k}}{\binom{n}{k}}+\frac{k-\ell}{n}=\frac{\lvert \G_0\rvert}{\binom{n}{k}}+\frac{k-\ell}{n},    
\end{equation}
and furthermore equality occurs if and only if either $\ell=k$, or $\ell<k$ and $x_0=n-1$. It then follows immediately from \eqref{eq:Katonacircle-bound}, \eqref{eq:Lovasz} and \eqref{eq:different-slices} that equality holds in \eqref{eq:contradiction}. For this, it remains to prove \eqref{eq:different-slices}, which amounts to showing that $x_0$ satisfies the conditions of \cref{lem:different-slices}.

As an intermediate step, we bound the size of $\G_0$ from below.
The following claim was proved in \cite{Frankl21}. For completeness, we also provide a proof here.
\begin{claim}[\cite{Frankl21}]\label{claim:lower-F0}
$\lvert \G_0 \rvert \ge (s+1)\binom{n-1}{k}-s\binom{n}{k}+\binom{ks}{k}$.
\end{claim}	
\begin{poc}
As $\G_0$ is non-empty, it contains some $G_0\in \binom{[n]}{k}$. Fix an arbitrary subset $X\subset [n]$ satisfying $\lvert X\rvert=ks$ and $G_0\cup X=[n]$.
For $1\le i\le s$, define $\mH_i=\G_i\cap \binom{X}{k}$. Notice that the families $\mH_1,\ldots,\mH_s$ are cross-union relative to $X$. Indeed, if $H_1\in \mH_1, \ldots, H_s \in \mH_s$ satisfy $H_1\cup \ldots \cup H_s=X$, then adding $G_0\in \G_0$ gives a contradiction to the cross-union property of $\G_0,\ldots,\G_s$. 

Applying
\cref{lem:Katonascircle_prob} to the $s$ families $\mH_1,\ldots,\mH_s \subset \binom{X}{k}$ yields
$\sum_{i=1}^s \lvert \mH_i \rvert  \le (s-1) \binom{ks}{k}$.
	So $$\sum_{i=1}^s \lvert \G_i \rvert \le \sum_{i=1}^s \left(\lvert \mH_i \rvert+\binom{n}{k} - \binom{ks}{k}\right) \le s\binom{n}{k}-\binom{ks}{k}.$$
	Together with \eqref{eq:contradiction} this gives $\lvert \G_0 \rvert \ge (s+1)\binom{n-1}{k}-s\binom{n}{k}+\binom{ks}{k}$, as desired. 
\end{poc}

\begin{claim}
$x_0$ meets the conditions of \cref{lem:different-slices}. In particular, $x_0$ 
satisfies \eqref{eq:different-slices}.
\end{claim}
\begin{poc}	
We know that $k\le x_0\le n-1$. It remains to show $\frac{\binom{x_0}{\ell}}{\binom{n}{\ell}}  \le \frac{k}{\ell} \frac{\binom{x_0}{k}}{\binom{n}{k}}$.
In order to do this, we distinguish two cases.

\textbf{Case 1: $k \ge 2 \ell$.}
It follows from \cref{claim:lower-F0} and \cref{lem:computation} (i) that
$\frac{ \binom{x_0}k}{\binom nk} \ge \frac{\ell}{k}.$ Moreover,
$\frac{ \binom{x_0}{\ell}}{\binom{n}{\ell}} < 1$ for $x_0 < n$. Hence
$\frac{\binom{x_0}{\ell}}{\binom{n}{\ell}}  < \frac{k}{\ell} \frac{\binom{x_0}{k}}{\binom{n}{k}}$.

\textbf{Case 2: $k < 2 \ell$.}
From \cref{claim:lower-F0} and \cref{lem:computation} (ii), we get $x_0\ge (1-1/k)n+1$. Hence
$$\frac{\binom{x_0}{k}}{\binom{n}{k}} =
	\frac{\binom{x_0}{\ell}}{\binom{n}{\ell}}\cdot\prod_{i=\ell}^{k-1} \frac{x_0-i}{n-i}
	\ge \frac{\binom{x_0}{\ell}}{\binom{n}{\ell}}\left(\frac{ x_0-k}{n-k}\right)^{k-\ell}
	\ge \frac{\binom{x_0}{\ell}}{\binom{n}{\ell}}\left( 1-\frac 1k \right)^{k-\ell}\ge \frac{\binom{x_0}{\ell}}{\binom{n}{\ell}}\left( 1-\frac {k-\ell}k \right)
	= \frac{\binom{x_0}{\ell}}{\binom{n}{\ell}}\frac{\ell}{k},$$
as required. Here the last inequality follows from Bernoulli's inequality.
\end{poc}

Therefore, as explained above, equality holds in~\eqref{eq:contradiction}. We now characterize $\mathcal{F}_0, \dots, \mathcal{F}_s$ for which equality holds in \cref{thr:main}. Equality in~\eqref{eq:contradiction} gives us $\sum_{i=0}^s\lvert\F_i\rvert = \sum_{i=0}^s\lvert\G_i\rvert=(s+1)\binom{n-1}{k}$, so we have equalities in \eqref{eq:Katonacircle-bound}, \eqref{eq:Lovasz} and \eqref{eq:different-slices}. Recall that equality occurs in \eqref{eq:different-slices} if and only if either $\ell=k$, or $\ell<k$ and $x_0=n-1$.

\begin{claim}\label{clm:equalfam}
    $\F_0=\F_1= \ldots =\F_s$.
\end{claim}

\begin{poc}
Suppose to the contrary that the families $\F_0,\ldots,\F_s$ are not all the same, say $\F_0\ne \F_1$.
If all the sizes are equal, i.e. $|\F_0|=\ldots=|\F_s|=\binom{n-1}{k}$, we replace $\F_0, \F_1$ by $\F_0 \cap \F_1, \F_0\cup \F_1$. Since $|\F_0|+|\F_1|=2\binom{n-1}{k}>\binom{n}{k}$ for $n>2k$, $\F_0 \cap \F_1$ is non-empty, and also the sum of sizes and the cross-union property are preserved. In addition, $|\F_0 \cap \F_1|<|\F_0\cup \F_1|$ for $\F_0\ne \F_1$.
Therefore, we can assume that $|\F_0|,\ldots,|\F_s|$ are not all equal.
\cref{lem:shifting} then tells us that $\G_0,\ldots,\G_s$ are not all the same. 

Since equality occurs in \eqref{eq:different-slices}, there are only two possibilities.


\textbf{Case 1:} $\ell<k$ and $x_0=n-1$. 
Since $\G_0,\ldots,\G_s$ are not all the same, we have
\[
\binom{x_0}{k}=|\G_0|<\frac{|\G_0|+\ldots+|\G_s|}{s+1}=\binom{n-1}{k}.
\]
This gives $x_0<n-1$, a contradiction.
    
    \textbf{Case 2: $ \ell=k$.}
    In this case, we need equality in Lemma~\ref{lem:Katonascircle_prob} for $k_0=\ldots=k_s=k$ and $n=(s+1)k$. We thus get $\G_0=\ldots=\G_s$, a contradiction.
\end{poc}

We learn from Claim~\ref{clm:equalfam} that $\F_0=\ldots =\F_s=\F$. Since $\{\F_0,\ldots,\F_s\}$ is cross-union and $\sum_{i=0}^s\lvert\F_i\rvert =(s+1)\binom{n-1}{k}$, we see that $\F$ is an $(s+1)$-wise union family of size $\binom{n-1}{k}$. Hence the uniqueness statement follows immediately from Theorem~\ref{thr:r-wise} (since $s+1>2$).
\end{proof}


%
\section{Concluding remarks}\label{sec:conc}

One remaining question is to determine the smallest value of $s_0$ for which \cref{conj:Frankl_crossunion} holds. As our theorem provides that this best value of $s_0$ is at most $4 \ell$ while the example at the introduction shows that it must be $\Omega\left(\frac{\ell}{\ln \ell}\right)$. It would be interesting to determine the correct order of magnitude for $s_0(\ell)$.

Another interesting question is what happens when $s$ is smaller than $s_0(\ell)$. In such a case, would \cref{examp:example} provide an extremal example? In particular, would the answer of the following question be true?

\begin{question}
	Let $n=ks+\ell$ with $0<\ell<k$, and let $\mathcal F_0, \mathcal F_1, \dots, \mathcal F_s \subset \binom{[n]}{k}$ be non-empty cross-union families. Does the following inequality hold?
	$$\sum_{i=0}^{s} \lvert \F_i \rvert \le \max \left \{ (s+1)\binom{n-1}{k}, 1+s\binom{n}{k}-\sum_{i=0}^{k-\ell} \binom{k}{i}\binom{n-k}{k-i}   \right \}$$
\end{question}

On the other hand, Conjecture~\ref{conj:Frankl_crossunion} motivates the `cross' version of the Erd\H{o}s matching conjecture as follows.

In \cite{FK21}, Frankl and Kupavskii defined that   families $\F_0,\dots, \F_{s}$ satisfy the property $U(s+1,q)$ if $\lvert F_0 \cup F_1 \cup \ldots \cup F_s \rvert \le q$  for every choice of $F_0 \in \F_0, \ldots, F_s \in \F_s$.
The condition of being cross-union is the same as having the property $U(s+1,n-1)$ and 
the condition on the Erd\H{o}s matching conjecture is the same as $\F_0=\dots=\F_{s+1}=\F$ having the property $U(s+1,k(s+1)-1)$.
This provides the natural `cross' version of the Erd\H{o}s matching conjecture by considering  the geometric mean and arithmetic mean of families satisfying the condition $U(s+1,k(s+1)-1)$.

For the maximum value of $\prod_{0\leq i\leq s}|\F_i|$ where $\F_0,\dots,\F_s$ have the property $U(s+1,k(s+1)-1)$, one can naturally consider $\F_0=\F_1= \{A \in \binom{[n]}{k}: 1\in A \}$ and $\F_2=\dots =\F_s =\binom{[n]}{k}$. In fact, the following proposition provides that this is an extremal example provided that $n$ is sufficiently large.

\begin{proposition}
For $k,s\geq 1$, there exists $n_0(k,s)$ such that the following holds for all $n\geq n_0(k,s)$. If $\mathcal F_0, \mathcal F_1, \dots, \mathcal F_s \subset \binom{[n]}{k}$ are non-empty families having the property $U(s+1,k(s+1)-1)$, then we have 
	$$\prod_{i=0}^s \lvert \F_i\rvert \le \binom{n-1}{k-1}^2\binom{n}{k}^{s-1}.$$
\end{proposition}
\noindent The result for $s=1$ is due to Pyber~\cite{Pyber86}. For $s \ge 2$, it is sufficient to note that for $n$ sufficiently large, 
$\left( \binom nk - \binom{n-ks}{k}\right)^{s+1}$ is smaller than the expression in the proposition.
If $\F_s$ is the largest family and the other $s$ families have $k$ pairwise disjoint sets, then all families have size at most $|\F_s|\leq  \binom nk - \binom{n-ks}{k}$ as desired. If this is not the case, then the result follows by induction on $s$.

On the other hand, it is interesting whether the above bound is actually best possible when $n$ is close to $ks$. For all we know,  $\binom{n-1}{k-1}^{s+1}$ can be the correct maximum when $n$ is just above $ks$.

For the maximum value of $\sum_{0\leq i\leq s}|\F_i|$, 
the families $\mathcal F_0=[k], \mathcal F_1 = \{ A \in \binom{[n]}{k} \colon \lvert A \cap [k] \vert \ge 1 \}$ and $\mathcal F_2= \ldots = \mathcal F_s = \binom{[n]}{k}$ are natural candidates for an extremal example.
The following proposition yields that indeed this is an extremal example for sufficiently large $n$. 

\begin{proposition}
For $k,s\geq 1$, there exists $n_0(k,s)$ such that the following holds for all $n\geq n_0(k,s)$.
	If $\mathcal F_0, \mathcal F_1, \dots, \mathcal F_s \subset \binom{[n]}{k}$ are non-empty families having the property $U(s+1,k(s+1)-1)$, then we have 
	$$\sum_{i=0}^s \lvert \F_i\rvert \le 1+s\binom{n}{k}-\binom{n-k}{k}.$$
\end{proposition}
\noindent The result for $s=1$ is due to Hilton and Milner~\cite{HiltonMilner67}, and a similar induction as before works as $(s+1)\left( \binom nk - \binom{n-ks}{k}\right)$ is smaller than the expression in the proposition for $n$ sufficiently large.

Even when $n=ks+\ell$ with small $\ell$, as long as $k>\ell$, the term $1+s\binom{n}{k}-\binom{n-k}{k}$ is bigger than $(s+1)\binom{ks-1}{k}$. Hence, the above example shows that, unlike Conjecture~\ref{conj:Frankl_crossunion},  $\F_0=\dots=\F_{s}=\binom{[ks-1]}{k}$ is not an extremal example when $n>k(s+1)$.

While the maximum of the geometric mean and the arithmetic mean of the families satisfying $U(s+1,k(s+1)-1)$ may behave differently from what is conjectured in the Erd\H{o}s matching conjecture, it has been conjectured~\cite{Ahpreprint, HLS12} that the minimum size behaves as in the Erd\H{o}s matching conjecture.

\begin{conjecture}[\cite{Ahpreprint, HLS12}]
If $n\ge k(s+1)$ and $\mathcal F_0, \mathcal F_1, \dots, \mathcal F_s \subset \binom{[n]}{k}$ are non-empty families such that $\lvert F_0 \cup F_1 \cup \ldots \cup F_s \rvert \le k(s+1)-1$ for every $F_0 \in \F_0, \ldots F_s \in \F_s$, then
	$$\min\left\{  \lvert \F_0 \rvert,  \lvert \F_1 \rvert, \ldots,  \lvert \F_s \rvert \right\}\le \max \left \{ \binom{n}{k}-\binom{n-s}{k}, \binom{k(s+1)-1}{k} \right \}.$$
\end{conjecture}
\noindent Recently, Kupavskii~\cite{Kupavskii21} proved this conjecture for $s>10^7$ and $n>3e(s+1)k$. 



\section*{Acknowledgements}

The authors would like to express their gratitude towards the referees for careful reading and suggestions to improve the paper, as well as suggestions for other references.


%
%

\bibliographystyle{alphaurl}
\bibliography{crossunionfam}

\end{document}